\documentclass[a4paper]{amsart}
\usepackage{geometry}
\usepackage{amsthm}
\usepackage{amssymb}
\usepackage{amsmath}
\usepackage{hyperref}
\usepackage{xcolor}
\usepackage{enumerate}
\usepackage{comment}
\usepackage{braket}

\newtheorem{theorem}{Theorem}
\newtheorem{definition}{Definition}
\newtheorem{lemma}[theorem]{Lemma}
\newtheorem{corollary}[theorem]{Corollary}

\newtheorem{conjecture}{Conjecture}
\newtheorem{problem}{Problem}
\newtheorem{claim}{Claim}

\newcommand{\mc}[1]{\mathcal{#1}}
\newcommand{\vv}{\textup{\textsf{v}}}
\newcommand{\e}{\textup{\textsf{e}}}

\author
{
Olivier Fischer 
\and
Raphael Steiner 
}
\thanks{Department of Computer Science, Institute of Theoretical Computer Science, ETH Z\"{u}rich, Switzerland,  \texttt{olivier.fischer@inf.ethz.ch, raphaelmario.steiner@inf.ethz.ch}.
The second author was supported by an ETH Postdoctoral Fellowship.}

\date{\today}

\title{On the choosability of $H$-minor-free graphs}

\begin{document}
\maketitle

\begin{abstract}
Given a graph $H$, let us denote by $f_\chi(H)$ and $f_\ell(H)$, respectively, the maximum chromatic number and the maximum list chromatic number of $H$-minor-free graphs. Hadwiger's famous coloring conjecture from 1943 states that $f_\chi(K_t)=t-1$ for every $t \ge 2$. A closely related problem that has received significant attention in the past concerns $f_\ell(K_t)$, for which it is known that $2t-o(t) \le f_\ell(K_t) \le O(t (\log \log t)^6)$. Thus, $f_\ell(K_t)$ is bounded away from the conjectured value $t-1$ for $f_\chi(K_t)$ by at least a constant factor. The so-called $H$-Hadwiger's conjecture, proposed by Seymour, asks to prove that $f_\chi(H)=\vv(H)-1$ for a given graph $H$ (which would be implied by Hadwiger's conjecture). 

In this paper, we prove several new lower bounds on $f_\ell(H)$, thus exploring the limits of a list coloring extension of $H$-Hadwiger's conjecture. Our main results are:

\begin{itemize}
    \item For every $\varepsilon>0$ and all sufficiently large graphs $H$ we have $f_\ell(H)\ge (1-\varepsilon)(\vv(H)+\kappa(H))$, where $\kappa(H)$ denotes the vertex-connectivity of $H$. 
    \item For every $\varepsilon>0$ there exists $C=C(\varepsilon)>0$ such that asymptotically almost every $n$-vertex graph $H$ with $\left\lceil C n\log n\right\rceil$ edges satisfies $f_\ell(H)\ge (2-\varepsilon)n$. 
\end{itemize}
The first result generalizes recent results on complete and complete bipartite graphs and shows that the list chromatic number of $H$-minor-free graphs is separated from the natural lower bound $(\vv(H)-1)$ by a constant factor for all large graphs $H$ of linear connectivity. The second result tells us that even when $H$ is a very sparse graph (with an average degree just logarithmic in its order), $f_\ell(H)$ can still be separated from $(\vv(H)-1)$ by a constant factor arbitrarily close to $2$.

Conceptually these results indicate that the graphs $H$ for which $f_\ell(H)$ is close to the conjectured value $(\vv(H)-1)$ for $f_\chi(H)$ are typically rather sparse.
\end{abstract}

\section{Introduction}
All graphs considered in this paper are finite, have no loops and no parallel edges. Given graphs $G$ and $H$, we say that $G$ contains $H$ as a \emph{minor}, in symbols, $G\succeq H$, if a graph isomorphic to $H$ can be obtained from a subgraph of $G$ by contracting edges.

Hadwiger's coloring conjecture, first stated in 1943 by Hugo Hadwiger~\cite{hadwiger}, is among the most famous and important open problems in graph theory. It claims a deep relationship between the chromatic number of graphs and their containment of graph minors, as follows.
\begin{conjecture}[Hadwiger~\cite{hadwiger}]
Let $t \in \mathbb{N}$. If a graph $G$ is $K_t$-minor-free, then $\chi(G)\le t-1$. 
\end{conjecture}

Hadwiger's conjecture has been proved for all values $t \le 6$, see~\cite{robertson} for the most recent result in this sequence, resolving the $K_6$-minor-free case. For $t=5$, the conjecture states that $K_5$-minor-free graphs are $4$-colorable. Since planar graphs are $K_5$-minor-free, this special case already generalizes the famous four color theorem that was proved in 1976 by Appel, Haken and Koch~\cite{appelhaken1,appelhaken2}. 

Given that during 80 years of study little progress has been made towards resolving Hadwiger's conjecture for $t \ge 7$, it seems natural to approach the conjecture via meaningful relaxations. For instance, much of recent work has focused on its asymptotic version. The so-called \emph{linear Hadwiger conjecture} states that for some absolute constant $C\ge 1$, every $K_t$-minor-free graph is $\lfloor Ct \rfloor$-colorable. Starting with a breakthrough result by Norin, Postle and Song~\cite{norine} in 2019, there has been a set of papers providing some exciting progress towards this conjecture~\cite{norine2,postle,postle2,postle3}. This culminated in the currently best known upper bound of $O(t \log\log t)$ for the chromatic number of $K_t$-minor-free graphs by Delcourt and Postle~\cite{del} in 2021.

Another natural relaxation, proposed by Seymour~\cite{survey,seymour}, suggests replacing the condition that the considered graphs exclude $K_t$ as a minor by the weaker condition that they exclude a specified, possibly non-complete graph $H$ on $t$ vertices as a minor. \begin{conjecture}[$H$-Hadwiger conjecture~\cite{survey,seymour}] $H$-minor-free graphs are $(\vv(H)-1)$-colorable.   
\end{conjecture} 
Note that Hadwiger's conjecture would imply the truth of this statement for every $H$. Also, note that this upper bound on the chromatic number would be best possible for every $H$, as the complete graph $K_{\vv(H)-1}$ has chromatic number $(\vv(H)-1)$ but is too small to host an $H$-minor. 

The $H$-Hadwiger conjecture can easily be verified using a degeneracy-coloring approach if $H$ is a forest, and it is also known to be true for spanning subgraphs of the Petersen graph~\cite{hendrey}. A particular case of the $H$-Hadwiger conjecture which has received special attention in the past is when $H=K_{s,t}$ is a complete bipartite graph. Woodall~\cite{woodall} conjectured in 2001 that every $K_{s,t}$-minor-free graph is $(s+t-1)$-colorable. Also this problem remains open, but if true it would resolve the $H$-Hadwiger conjecture for all bipartite $H$. Several special cases of this conjecture have been solved by now. Most notably, Kostochka~\cite{kostochkabip1,kostochkabip2} proved that for some function $t_0(s)=O(s^3\log^3 s)$, the $H$-Hadwiger conjecture holds whenever $H=K_{s,t}$ and $t \ge t_0(s)$. The conjecture is also true for $H=K_{3,3}$, which can be seen using the structure theorem for $K_{3,3}$-minor-free graphs by Wagner~\cite{wagner} and the fact that planar graphs are $5$-colorable. In addition, the statement has been proved for $H=K_{2,t}$ when $t \ge 1$~\cite{chudnovsky,myers,woodall,woodall2}, for $H=K_{3,t}$ when $t \ge 6300$~\cite{kostochkabip3} and for $H=K_{3,4}$~\cite{jorgensen}. In a different direction, Norin and Turcotte~\cite{norine4} recently proved the $H$-Hadwiger conjecture for all sufficiently large bipartite graphs of bounded maximum degree that belong to a class of graphs with strongly sublinear separators. 

\subsection*{List coloring $H$-minor-free graphs} In this paper, we shall be concerned with the \emph{list chromatic number} of graphs that exclude a fixed graph $H$ as a minor. List coloring is a well-known and popular subject in the area of graph coloring, whose introduction dates back to the seminal paper of Erd\H{o}s, Rubin and Taylor~\cite{erdos}. A \emph{list assignment} for a graph $G$ is a mapping $L:V(G)\rightarrow 2^\mathbb{N}$ assigning to every vertex $v \in V(G)$ a finite set $L(v)$ of colors, also called the \emph{list} of $v$. An \emph{$L$-coloring} of $G$ is a proper coloring $c:V(G)\rightarrow \mathbb{N}$ for which every vertex must choose a color from its list, that is, $c(v) \in L(v)$ for every $v \in V(G)$. Finally, we say that $G$ is \emph{$k$-choosable} for some integer $k \ge 1$ if there exists a proper $L$-coloring for every list-assignment $L$ satisfying $|L(v)|\ge k$ for all $v \in V(G)$. The \emph{list chromatic number} of $G$, denoted $\chi_\ell(G)$, is the smallest integer $k$ such that $G$ is $k$-choosable. Note that trivially $\chi(G)\le \chi_\ell(G)$ for every graph $G$, but conversely no relationship holds, as $\chi_\ell(G)$ is unbounded even on bipartite graphs $G$, see~\cite{erdos}. 

The first open problem regarding list coloring of minor-closed graph classes was raised already in the seminal paper~\cite{erdos} by Erd\H{o}s et al. in 1979, who asked to determine the maximum list chromatic number of planar graphs. This question was answered in the 1990s in work of Thomassen~\cite{thomassen} and Voigt~\cite{voigt}. Thomassen proved that every planar graph is $5$-choosable, and Voigt gave the first examples of planar graphs $G$ with list chromatic number $\chi_\ell(G)=5$. 

The latter result also answered a question by Borowiecki~\cite{borowiecki} in the negative, who had asked whether one could potentially strengthen Hadwiger's conjecture to the list coloring setting by asserting that every $K_t$-minor-free graph $G$ satisfies $\chi_\ell(G)\le t-1$. 

Given the previous discussion, it is natural to study the maximum list chromatic number of $K_t$-minor-free graphs, see also~\cite{op} for an open problem garden entry about this problem. To make the following presentation more convenient, for every graph $H$ we denote by $f_\chi(H)$ and $f_\ell(H)$, respectively, the maximum (list) chromatic number of $H$-minor-free graphs. Note that with this notation, the $H$-Hadwiger conjecture amounts to saying that $f_\chi(H)=\vv(H)-1$.

Let us briefly summarize previous work regarding bounds on $f_\ell(K_t)$. The construction of Voigt mentioned above shows that $f_\ell(K_5)\ge 5$. Thomassen's result regarding the $5$-choosability of planar graphs was later extended by \v{S}krekovski~\cite{skrekovski} to $K_5$-minor-free graphs, thus proving that $f_\ell(K_5)=5$. Until today none of the values $f_\ell(K_t)$ with $t \ge 6$ have been determined precisely, a list of the currently best known lower and upper bounds for $f_\ell(K_t)$ for small values of $t$ can be found in~\cite{barat}. In 2007, Kawarabayashi and Mohar~\cite{kawarabayashimohar} made two conjectures regarding the asymptotic behaviour of $f_\ell(K_t)$, namely that (A) $f_\ell(K_t)=O(t)$, this is known as the \emph{list linear Hadwiger conjecture}, and that (B) $f_\ell(K_t)\le \frac{3}{2}t$ for every $t$. In 2010, Wood~\cite{wood}, inspired by the fact that $f_\ell(K_5)=5$, proposed an even stronger conjecture stating that $f_\ell(K_t)=t$ for every $t \ge 5$. This strong conjecture was refuted in 2011 by Bar\'{a}t, Joret and Wood, who gave a construction showing that $f_\ell(K_t) \ge \frac{4}{3}t-O(1)$. However, the weaker conjecture (B) by Kawarabayashi and Mohar still remained open. Recently, a new lower bound of $f_\ell(K_t) \ge 2t-o(t)$ was established by the second author~\cite{steiner}, thus refuting conjecture~(B). As for upper bounds, the best currently known bound is $f_\ell(K_t) \le Ct (\log \log t)^6$, which was established in 2020 by Postle~\cite{postle3}. 
Some previous work also addressed bounds on $f_\ell(H)$ when $H$ is non-complete. In particular, Woodall~\cite{survey} conjectured in 2001 that $f_\ell(K_{s,t})=s+t-1$ for all integers $s, t \ge 1$, and proved this in the case when $s=t=3$. From the previously mentioned works~\cite{chudnovsky,myers,woodall,woodall2} it was also known that $f_\ell(K_{2,t})=t+1$ for $t \ge 1$. Additionally, a result by J\o{}rgensen~\cite{jorgensen} implied the truth of the conjecture for $K_{3,4}$, and Kawarabayashi~\cite{kawarabayashi} proved that $f_\ell(K_{4,t})\le 4t$ for every~$t$.  Despite this positive evidence, Woodall's conjecture was recently disproved by the second author~\cite{steiner2} showing that $f_\ell(K_{s,t})\ge (1-o(1))( 2s+t)$ for all large values of $s \le t$. 
A positive result comes from the aforementioned result of Norin and Turcotte~\cite{norine4}, which also works for list colorings and shows that $f_\ell(H)=\vv(H)-1$ for all large bipartite graphs $H$ of bounded maximum degree in a graph class with strongly sublinear separators.

\subsection*{Our contribution} The above discussion shows that when excluding a sufficiently large complete or a sufficiently large balanced complete bipartite graph $H$, the value of $f_\ell(H)$ exceeds the trivial lower bound $f_\ell(H) \ge \vv(H)-1$ by at least a constant factor. This means that, in a strong sense, one cannot hope for extending Hadwiger's conjecture to list coloring with the same quantitative bounds. However, note that if $H$ is a complete or a balanced complete bipartite graph, then $H$ is quite dense in the sense that it has a quadratic number of edges. On the other extreme of the spectrum, the previously mentioned result by Norin and Turcotte~\cite{norine4} shows that $f_\ell(H)=\vv(H)-1$ \emph{does} hold for large classes of graphs $H$ with a constant maximum degree (and thus, with a linear number of edges). This naturally opens up a new question, as follows: How sparse must the desired minor $H$ be, such that one \emph{can} hope for a list coloring extension of the $H$-Hadwiger conjecture? Concretely, which structural and density properties of graphs $H$ guarantee that $f_\ell(H)=\vv(H)-1$? 
While one might be tempted to hope for a nice description of the class of all graphs $H$ satisfying $f_\ell(H)=\vv(H)-1$, Theorem~\ref{prop:isol} below speaks a word of caution: Any given graph $F$ can be augmented, by the addition of sufficiently many isolated vertices, to a graph $H$ in this class. 

\begin{theorem}\label{prop:isol}
For every graph $F$ there exists $k_0=k_0(F)$ such that for every $k \ge k_0$ the graph $H$ obtained from $F$ by the addition of $k$ isolated vertices satisfies $f_\ell(H)=\vv(H)-1$. In fact, every $H$-minor-free graph is $(\vv(H)-2)$-degenerate.
\end{theorem}

This shows that arbitrary graphs $F$ can show up as subgraphs\footnote{and in fact, as induced subgraphs} of graphs $H$ with $f_\ell(H)=\vv(H)-1$. To avoid such artifical constructions and to make a nice structural description of the graph class at hand more likely, it seems natural to ask for the largest class that is \emph{closed under taking subgraphs} such that all members $H$ of this class satisfy $f_\ell(H)=\vv(H)-1$. \footnote{This is done in the spirit of the definition of perfect graphs, where a nice characterization of graphs with $\chi(G)=\omega(G)$ seems elusive, but the largest class of graphs with this property that is closed under taking induced subgraphs admits a beautiful structural description by the strong perfect graph theorem.}

\begin{problem}\label{prob:perf}
Characterize the class $\mathcal{H}$ of graphs $H$ such that $f_\ell(H')=\vv(H')-1$ for all $H'\subseteq H$. 
\end{problem}

The main contributions of this paper are Theorems~\ref{thm:conn} and ~\ref{thm:random} below, which establish new lower bounds on $f_\ell(H)$ and strongly limit the horizon for positive instances of Problem~\ref{prob:perf}. The first result proves a lower bound on $f_\ell(H)$ in terms of $\vv(H)$ and the vertex-connectivity $\kappa(H)$, implying that $f_\ell(H)$ exceeds $\vv(H)$ by a constant factor for all large graphs of linear connectivity\footnote{Here, by graphs of linear connectivity we mean $n$-vertex graphs $H$ that are $\alpha n$-connected for some small but absolute constant $\alpha>0$.}.

\begin{theorem}\label{thm:conn}
For every $\varepsilon>0$ there exists $n_0=n_0(\varepsilon)\in \mathbb{N}$ such that every graph $H$ on at least $n_0$ vertices satisfies $f_\ell(H)\ge (1-\varepsilon)(\vv(H)+\kappa(H))$. 
\end{theorem}

In particular, this result immediately generalizes both of the lower bounds of $f_\ell(K_t)\ge 2t-o(t)$ and $f_\ell(K_{s,t})\ge (1-o(1))(2s+t)$ previously established by the second author in~\cite{steiner,steiner2} by noting that $\kappa(K_t)=t-1$ and $\kappa(K_{s,t})=s$ for $s \le t$. It also has the following simple consequence, showing that the graphs in $\mathcal{H}$ have a subquadratic number of edges.

\begin{corollary}\label{cor:edges}
For every $n \in \mathbb{N}$, let $h(n)$ denote the maximum possible number of edges of an $n$-vertex graph in $\mathcal{H}$. Then $\lim_{n\rightarrow \infty}\frac{h(n)}{n^2}=0$. 
\end{corollary}
\begin{proof}
Towards a contradiction, suppose the statement is not true. Then there is some $\delta \in (0,1)$ and an infinite strictly increasing sequence $(n_i)_{i=1}^\infty$ of natural numbers such that $\frac{h(n_i)}{n_i^2}\ge \delta$ for all $i \ge 1$. This means that for every $i$ there exists $H_i \in \mathcal{H}$ of order $n_i$ with $\e(H_i) \ge \delta n_i^2$ edges. By a classical result of Mader~\cite{mader} (see also~\cite{carmesin} for an optimal constant), every graph of average degree at least $4(k-1)$ for some integer $k \ge 2$ contains a $k$-connected subgraph. Consequently, for every $i$ there exists a subgraph $H_i'\subseteq H_i$ such that $\kappa(H_i')\ge\frac{1}{4}d_i$, where $d_i$ denotes the average degree of $H_i$. Since $\mathcal{H}$ is closed under taking subgraphs, we have $H_i' \in \mathcal{H}$. And as $d_i=\frac{2\e(H_i)}{\vv(H_i)} \ge \frac{2\delta n_i^2}{n_i}=2\delta n_i$, we conclude that $\kappa(H_i')\ge \frac{\delta}{2}{n_i}\ge \frac{\delta}{2}{\vv(H_i')}$ for every $i$. Furthermore, we have $\vv(H_i')> \kappa(H_i')\ge \frac{\delta}{2}{n_i} \rightarrow \infty$ for $i \rightarrow \infty$. The latter implies that there exists some $j \in \mathbb{N}$ such that $\vv(H_j')>n_0$, where $n_0=n_0(\varepsilon)$ with $\varepsilon:=\frac{\delta}{4}$ is the constant as given by Theorem~\ref{thm:conn}. The statement of Theorem~\ref{thm:conn} now implies that $f_\ell(H_j')\ge (1-\varepsilon)(\vv(H_j')+\kappa(H_j'))$. Consequently, $f_\ell(H_j') \ge (1-\frac{\delta}{4})(\vv(H_j')+\frac{\delta}{2}\vv(H_j'))=(1+\frac{\delta}{4}-\frac{\delta^2}{8})\vv(H_j')>\vv(H_j')$. Together with $f_\ell(H_j')=\vv(H_j')-1$ which follows by definition of $\mathcal{H}$, this yields the desired contradiction and concludes the proof.
\end{proof}

Our second result addresses to what extent sparsity of $H$ can push $f_\ell(H)$ closer to the trivial lower bound $\vv(H)-1$, by showing that for any fixed $\varepsilon>0$, asymptotically almost all $n$-vertex graphs $H$ with average degree of order $C\log n$ for a sufficiently large constant $C$ are far from being in $\mathcal{H}$, in the sense that $f_\ell(H)$ is separated from $\vv(H)-1$ by a factor of at least $2-\varepsilon$. 

\begin{theorem}\label{thm:random}
For every $\varepsilon>0$ there exists a constant $C=C(\varepsilon)>0$ such that asymptotically almost every graph $H$ on $n$ vertices with $\lceil C n\log n \rceil$ edges satisfies $f_\ell(H)\ge (2-\varepsilon) n$. 
\end{theorem}

This is conceptually stronger than what is guaranteed by Corollary~\ref{cor:edges}, as it indicates that graphs in $\mathcal{H}$ typically have as few as $O(n\log n)$ edges. It also shows that the lower bound $f_\ell(K_t)\ge 2t-o(t)$ for complete graphs from~\cite{steiner} applies in equal strength to almost all $t$-vertex graphs $H$ with $\omega(t \log t)$ edges, despite them being (much) sparser than $K_t$.

Our proofs of Theorem~\ref{thm:conn} and Theorem~\ref{thm:random} are based on several extensions and refinements of the probabilistic approach for lower-bounding $f_\ell(K_t)$ and $f_\ell(K_{s,t})$ introduced by the second author in~\cite{steiner,steiner2}. However, several new ideas are required to overcome obstacles arising from the largely increased generality of the setup. For instance, to prove Theorem~\ref{thm:random} one has to construct graphs avoiding rather sparse graphs $H$ as a minor. While the constructions in~\cite{steiner,steiner2} were based on the fact that clique sums of graphs under mild assumptions preserve $K_t$- and $K_{s,t}$-minor-freeness, a corresponding statement is no longer true for sparse graphs $H$ of much lower connectivity.

\subsection*{Organization of the paper} In Section~\ref{sec:probab} we prove two probabilistic results on random bipartite graphs that exhibit properties of these graphs that are crucial for our constructions in the proofs of Theorems~\ref{thm:conn} and~\ref{thm:random}. We then present the proofs of our main results Theorem~\ref{thm:conn} and Theorem~\ref{thm:random} in, respectively, Section~\ref{sec:thm1proof} and Section~\ref{sec:thm2proof}. Finally, in Section~\ref{sec:append} we separately prove Theorem~\ref{prop:isol}. The latter proof is self-contained and independent of the results in the other three sections.

\subsection*{Notation and terminology} Given an integer $k \ge 1$, we denote by $[k]:=\{1,2,\ldots,k\}$ the set of integers from $1$ up to $k$.
For a graph $G$, we denote by $V(G)$ its vertex-set, by $E(G)$ its edge-set and by $\vv(G)$ and $\e(G)$ the respective cardinalities. For a vertex $v \in V(G)$ we denote by $N_G(v)$ its open neighborhood and by $\text{deg}_G(v)=|N_G(v)|$ its degree, where we occasionally drop the subscript $G$ if it is given by context. For two disjoint sets of vertices $A, B \subseteq V(G)$, we denote $E_G(A,B):=\{ab\in E(G)|(a,b) \in A \times B\}$ for the set of edges connecting $A$ and $B$, and we denote by $\e_G(A,B):=|E_G(A,B)|$ the number of such edges. We say that a graph $G$ is $k$-connected for an integer $k \ge 1$ if $\vv(G)\ge k+1$ and $G-X$ is connected for every $X\subseteq V(G)$ with $|X|<k$. By $\kappa(G)$ we denote the \emph{vertex-connectivity} of $G$, i.e., the minimum $k$ such that $G$ is $k$-connected. Given a random event $A$, we use $\mathbb{P}(A)$ to denote its probability, and we denote by $\mathbb{E}(X)$ the expectation of a real-valued random variable $X$.  Given a random model ranging on $n$-vertex graphs, we say that the outcome of this model satisfies a property \emph{with high probability} (abbreviation: w.h.p.) if the probability of the event that the random graph fulfils the desired property tends to $1$ as $n$ tends to infinity. Similarly, if $\mathcal{G}_n$ for every $n\in \mathbb{N}$ is a class of $n$-vertex graphs, we say that \emph{asymptotically almost every} graph $G$ in $\mathcal{G}_n$ satisfies a specified property if w.h.p. a random graph sampled uniformly at random from $\mathcal{G}_n$ satisfies this property. 
Given integers $m, n \ge 1$ and an edge-probability $p \in [0,1]$, we use $G(m,n;p)$ to denote the bipartite Erd\H{o}s-R\'{e}nyi random graph with bipartition classes $A$ and $B$ of sizes $m$ and $n$, respectively, and in which a pair $ab$ with $a\in A$ and $b \in B$ is chosen as an edge of $G(m,n;p)$ independently with probability $p$. For integers $m, n \ge 1$ we denote by $G(n;m)$ a random graph drawn uniformly from all graphs on vertex-set $[n]=\{1,\ldots,n\}$ with exactly $m$ edges. 

While the original definition of the graph minor-containment relation $\succeq$ is via edge contractions and deletions, for proving the results in this paper it will be more convenient to think about \emph{graph minor models}. Given a graph $G$ and a graph $H$, an \emph{$H$-minor model} is a collection $(Z_h)_{h \in V(H)}$ of pairwise disjoint and non-empty subsets of $V(G)$ with the property that $G[Z_h]$ is a connected graph for every $h \in V(H)$ and such that for every edge $h_1h_2 \in E(H)$, there exists at least one edge in $G$ with endpoints in $Z_{h_1}$ and $Z_{h_2}$. The sets $Z_h, h \in V(H)$ are also called  the \emph{branch sets} of the minor model. 
It is well-known and easy to see that for every pair of graphs $G$ and $H$ we have $G \succeq H$ if and only if there exists an $H$-minor model in $G$. 

\section{Probabilistic lemmas}\label{sec:probab}
In this short preparatory section we prove two simple auxiliary results (Lemmas~\ref{lemma:maxdeg} and~\ref{lemma:randbip}) that will be used in the proofs of both our main results in Section~\ref{sec:thm1proof} and~\ref{sec:thm2proof}. The lemmas capture two simple but important properties exhibited by bipartite Erd\H{o}s-R\'{e}nyi random graphs. These properties will later be used to lower-bound the list chromatic number of the graphs in our constructions for Theorem~\ref{thm:conn} and Theorem~\ref{thm:random} and to argue that they exclude a given graph as a minor.

A basic tool from probability theory that we will use in the following are the classical Chernoff concentration bounds for binomially distributed random variables, stated below.
\begin{lemma}[Chernoff] \label{lemma:chernoff}
    Let $X$ be a binomially distributed random variable. Then the following bounds hold for every $\delta \in (0,1]$:
    $$\mathbb{P}(X\ge (1+\delta)\mathbb{E}(X)) \le \exp\left(-\frac{\delta^2}{3}\mathbb{E}(X)\right), \ \mathbb{P}(X\le (1-\delta)\mathbb{E}(X)) \le \exp\left(-\frac{\delta^2}{2}\mathbb{E}(X)\right).$$
\end{lemma}

\begin{lemma} \label{lemma:maxdeg}
Let $p \in (0,1]$ be a constant. 
Then w.h.p. the random bipartite graph $G=G(n, n; p)$ has maximum degree at most $2pn$. 
\end{lemma}
\begin{proof}
Let $A,B$ be the vertex bipartition of $G$. 
Let $\mathcal{E}$ denote the event that some vertex $v \in A \cup B$ has $\deg (v) > 2p n$. 
For each $v \in A \cup B$, the number of neighbors is distributed binomially according to $\deg (v) \sim B(n, p)$, so we can apply the Chernoff bounds (Lemma~\ref{lemma:chernoff}). This yields for every $v \in A \cup B$: 
$$\mathbb{P}(\deg(v)\ge 2pn) = \mathbb{P}(\deg(v)\ge 2\mathbb{E}(\deg(v)))\le \exp\left(-\frac{1}{3}pn\right).$$
Using a union bound over all $v \in A \cup B$, we find that $\mathbb{P}(\mathcal{E})\le 2n \exp(-\frac{1}{3}pn)=o(1)$. Hence, w.h.p. the random bipartite graph $G$ has maximum degree at most $2pn$, as claimed.
\end{proof}

In order to compactly state and refer to our next lemma below, it is convenient for us to introduce the following technical definition.

\begin{definition}[Property $\textsf{P}$] \label{def:propQ}
Let $0<\delta < 1$, $s \in \mathbb{N}$ and let $H$ be a graph on $n$ vertices. We say that a bipartite graph $G$ with  bipartition $\{A, B\}$ satisfies property $\textsf{P}(H, \delta,s)$ if for all integers $k, l \geq \delta n$ the following holds:

If $x_1,\ldots,x_k, y_1, \ldots, y_l \in V(H)$ are distinct vertices satisfying $\e_H(\{x_1,\ldots,x_k\},\{y_1,\ldots,y_l\}) \ge s$ and $X_1, ..., X_{k} \subseteq A$, $Y_1, ..., Y_{l} \subseteq B$ are pairwise disjoint sets of size at most $\frac{1}{\delta}$ each, then there exists an index pair $(i,j) \in [k]\times [l]$ such that $x_iy_j \in E(H)$ and $xy \in E(G)$ for every $(x,y) \in X_i \times Y_j$. 
\end{definition}
\begin{lemma}\label{lemma:randbip}
Let $\delta, p \in (0,1)$ be constants. Then there exists a constant $D=D(\delta,p)>1$ and a sequence $q_n=1-o(1)$ such that with $s=s(n):=\lceil D n \log n \rceil$ for every $n$-vertex graph $H$ the random bipartite graph 
$G = G (n,n;p)$ satisfies $\textsf{P}(H, \delta,s)$ with probability at least $q_n$.
\end{lemma}
\begin{proof}
Let $D>1$ be chosen as any constant strictly bigger than $4p^{-(1/\delta^2)}$. Let $A, B$ be the vertex bipartition of $G$ with $|A| = |B| = n$, and let $H$ be an $n$-vertex graph. To enable a concise presentation, for every pair of integers $k, l \ge \delta n$ we denote by 
$$\mathcal{V}_{k,l}:=\Set{((x_i)_{i=1}^k,(y_j)_{j=1}^l) | \{x_1,\ldots,x_k,y_1,\ldots,y_l\} \subseteq V(H), \ \e_H(\{x_1,\ldots,x_k\},\{y_1,\ldots,y_l\})\ge s}$$ 
the set of admissible choices of vertices in the definition of property \textsf{P}, and by 
$$\mathcal{S}_{k,l}:=\Set{((X_i)_{i=1}^k,(Y_j)_{j=1}^l) |  X_1,\ldots,X_k\subseteq A,Y_1,\ldots,Y_l\subseteq B\text{ pairwise disjoint}, \ |X_i|, |Y_j|\le 1/\delta}$$ 
the set of admissible choices of the sets in the definition of property \textsf{P}. 
For each choice of $((x_i)_{i=1}^k,(y_j)_{j=1}^l, (X_i)_{i=1}^k,(Y_j)_{j=1}^l) \in \mathcal{V}_{k,l}\times \mathcal{S}_{k,l}$ let us denote by $\mc{E}((x_i)_{i=1}^k,(y_j)_{j=1}^{l},(X_i)_{i=1}^k, (Y_j)_{j=1}^{l})$ the random event that for every pair $(i,j) \in [k]\times [l]$ such that $x_iy_j \in E(H)$, not all of the potential edges between $X_i$ and $Y_j$ are included in $G$.

 Note that $G$ satisfies $\textsf{P}(H,\delta,s)$ if and only if the event $\mc{E}((x_i)_{i=1}^k,(y_j)_{j=1}^{l},(X_i)_{i=1}^k, (Y_j)_{j=1}^{l})$ does \textit{not} occur for any choice of $((x_i)_{i=1}^k,(y_j)_{j=1}^l, (X_i)_{i=1}^k,(Y_j)_{j=1}^l) \in \mathcal{V}_{k,l}\times \mathcal{S}_{k,l}$. Let us now bound the probability of a single event $\mc{E}((x_i)_{i=1}^k,(y_j)_{j=1}^{l},(X_i)_{i=1}^k, (Y_j)_{j=1}^{l})$. For each pair $(i,j)$ such that $x_iy_j \in E(H)$, the probability that all of the $|X_i||Y_j|$ potential edges between $X_i$ and $Y_j$ are present in $G$ is precisely $p^{|X_i||Y_j|}$. Since the edges of $G$ are sampled independently and since the sets of potential edges between $X_i$ and $Y_j$ and between $X_{i'}$ and $Y_{j'}$ for distinct pairs $(i,j), (i',j')$ are disjoint, we conclude that
$$\mathbb{P}(\mc{E}((x_i)_{i=1}^k,(y_j)_{j=1}^{l},(X_i)_{i=1}^k, (Y_j)_{j=1}^{l}))=\prod_{x_iy_j \in E(H)}(1-p^{|X_i||Y_j|}).$$ Using that all the sets $X_1,\ldots,X_k,Y_1,\ldots,Y_l$ are of size at most $\frac{1}{\delta}$ and that there are at least $s\ge Dn\log n$ edges of the form $x_iy_j\in E(H)$, it follows that
$$\mathbb{P}(\mc{E}((x_i)_{i=1}^k,(y_j)_{j=1}^{l},(X_i)_{i=1}^k, (Y_j)_{j=1}^{l}))\le (1-p^{(1/\delta^2)})^{Dn\log n} \le \exp(-p^{(1/\delta^2)}Dn\log n).$$ 

Note that in order to check that the graph $G$ satisfies $\textsf{P}(H,\delta,s)$, it suffices to verify the conditions for the smallest possible values of $k$ and $l$, i.e., when $k=l=\lceil \delta n\rceil$. Hence, using a union bound over all choices of $((x_i)_{i=1}^{\lceil \delta n\rceil},(y_j)_{j=1}^{\lceil \delta n\rceil},(X_i)_{i=1}^{\lceil \delta n\rceil}, (Y_j)_{j=1}^{\lceil \delta n\rceil}) \in \mathcal{V}_{\lceil \delta n\rceil,\lceil \delta n\rceil}\times \mathcal{S}_{\lceil \delta n\rceil,\lceil \delta n\rceil}$, we now obtain:
$$\mathbb{P}(G\text{ does not satisfy property }\textsf{P}(H,\delta,s))\le |\mathcal{V}_{\lceil \delta n\rceil,\lceil \delta n\rceil}| \cdot |\mathcal{S}_{\lceil \delta n\rceil,\lceil \delta n\rceil}|\cdot \exp(-p^{(1/\delta^2)}Dn\log n)$$
$$\le n^{2\lceil \delta n\rceil} \cdot (\lceil \delta n\rceil+1)^{|A|}\cdot (\lceil \delta n\rceil+1)^{|B|}\cdot \exp(-p^{(1/\delta^2)}Dn\log n)$$
$$= \exp(2\lceil\delta n\rceil\log n+2n\log(\lceil\delta n\rceil+1)-p^{(1/\delta^2)}Dn\log n)$$
$$\le \exp((4-p^{(1/\delta^2)}D)\cdot n\log n).$$
In the first estimate above, we upper-bounded the size of $|\mathcal{S}_{\lceil \delta n\rceil,\lceil \delta n\rceil}|$ by noting that the collections $(X_i)_{i=1}^{k}$ of disjoint subsets of $A$ are in one-to-one correspondence with the mappings from $A$ to $[k+1]$, and similarly for the set of collections $(Y_j)_{j=1}^{l}$.

By our choice of $D$, we have $4-p^{(1/\delta^2)}D<0$ and thus the above expression tends to $0$ as $n\rightarrow \infty$. Setting $q_n:=1-\exp((4-p^{(1/\delta^2)}D)n\log n)$ then concludes the proof of the lemma.
\end{proof}

\smallskip
\section{Proof of Theorem~\ref{thm:conn}}\label{sec:thm1proof}
In this section, we present the proof of Theorem~\ref{thm:conn}. We start off by making use of Lemmas~\ref{lemma:maxdeg} and ~\ref{lemma:randbip} from the previous section to establish the existence of small $H$-minor-free graphs that are in a sense ``almost complete'', as follows. 

\begin{lemma}\label{lemma:deterministiccorollary}
For every $\varepsilon \in (0,\frac{1}{2})$ there exists an integer $N=N(\varepsilon)$ such that for every $n \ge N$ and every $n$-vertex graph $H$ with $\kappa(H)\ge \varepsilon n$ there exists a graph $F$ with the following properties:
\begin{itemize}
    \item The vertex-set of $F$ can be partitioned into two disjoint sets $A$ and $B$ such that both $A$ and $B$ form cliques in $F$ and $|A|=\lfloor(1-2\varepsilon)\kappa(H)\rfloor$, $|B|=\lfloor(1-2\varepsilon)n\rfloor$.
    \item Every vertex in $B$ has at most $\varepsilon n$ non-neighbors in $F$.
    \item $F$ is $H$-minor-free.
\end{itemize}
\end{lemma}
\begin{proof}
Define $p:=\frac{\varepsilon}{2}$ and $\delta:=\varepsilon^2$. By Lemma~\ref{lemma:maxdeg} there is a sequence $p_n=1-o(1)$ such that $G(n,n;p)$ has maximum degree at most $2pn=\varepsilon n$ with probability at least $p_n$, and by Lemma~\ref{lemma:randbip} there exists an absolute constant $D>0$ and a sequence $q_n=1-o(1)$ such that for every $n$-vertex graph $H$ the probability that $G(n,n;p)$ satisfies property $\textsf{P}(H, \delta, \lceil Dn\log n\rceil)$ is at least $q_n$. Let $n_1$ be such that $p_n, q_n>\frac{1}{2}$ for every $n \ge n_1$. Moreover, let $n_2 \in \mathbb{N}$ be chosen large enough such that the inequality $\delta^2 n^2\ge D n \log n$ holds for every $n \ge n_2$. Finally, we put $N:=\max\{n_1,n_2\}$ and claim that the statement of the lemma holds for this choice of $N$. 

To prove this, let any integer $n \ge N$ and a graph $H$ on $n$ vertices with $\kappa(H) \ge \varepsilon n$ be given. Since $n \ge N \ge n_1$, we then have $p_n+q_n>1$, which implies that with positive probability, the random bipartite graph $G(n,n;p)$ \emph{simultaneously} has maximum degree at most $\varepsilon n$ and satisfies property $\textsf{P}(H, \delta, \lceil Dn\log n\rceil)$. Hence there exists at least one bipartite graph $G$ with bipartition $(A',B')$ such that $|A'|=|B'|=n$, $G$ has maximum degree at most $\varepsilon n$, and $G$ satisfies property $\textsf{P}(H,\delta, \lceil Dn\log n\rceil)$. Let $A \subseteq A', B \subseteq B'$ be chosen (arbitrarily) such that $|A|=\lfloor (1-2\varepsilon)\kappa(H)\rfloor$, $|B|=\lfloor(1-2\varepsilon)n\rfloor$. Note that this is possible as $\kappa(H)<\vv(H)=n$. We now define $F$ as the graph complement of the induced subgraph $G[A \cup B]$ of $G$. Since $A$ and $B$ are independent sets in $G$, they form cliques in $F$. Thus the first item of the lemma is satisfied. To verify the second item, it suffices to note that since $G$ has maximum degree at most $\varepsilon n$, the same is true for $G[A \cup B]$, and thus every vertex in $F$ can have at most $\varepsilon n$ non-neighbors in $F$. It thus remains to prove that $F$ is indeed $H$-minor-free. Towards a contradiction, suppose that there exists an $H$-minor model $(Z_h)_{h \in V(H)}$ in $F$. Let $X_A, X_B, X_{AB}$ be the partition of $V(H)$ defined as 
$$X_A:=\{h \in V(H)|Z_h \subseteq A\}, X_B:=\{h \in V(H)|Z_h \subseteq B\}, X_{AB}:=\{h \in V(H)|Z_h \cap A \neq \emptyset \neq Z_h \cap B\}.$$ 
Note that we have $|X_B|+|X_{AB}|\le |B|\le (1-2\varepsilon)n$ as the sets in $(Z_h)_{h \in V(H)}$ are pairwise disjoint. Given that $|X_A|+|X_B|+|X_{AB}|=\vv(H)=n$, this implies that $|X_A| \ge 2\varepsilon n$. Since the sets $(Z_h)_{h \in X_A}$ are disjoint and since $|A|\le (1-2\varepsilon)\kappa(H)<(1-2\varepsilon)n< n$, there cannot be more than $\delta n$ sets of size greater than $\frac{1}{\delta}$ in the collection $(Z_h)_{h \in X_A}$. Hence, there exists $k \ge 2\varepsilon n-\delta n\ge \delta n$ and $k$ distinct vertices $x_1, \ldots,x_k \in X_A$ such that $|Z_{x_i}| \le \frac{1}{\delta}$ for $i=1,\ldots,k$. Note that $H$ has minimum degree at least $\delta(H)\ge \kappa(H)$, for otherwise one could separate a vertex in $H$ from the rest of the graph by deleting fewer than $\kappa(H)$ vertices. Using this, we have $$|N_H(x_i) \cap X_B|\ge \deg_H(x_i)-|X_A \cup X_{AB}| \ge \delta(H)-|A|$$ $$\ge \kappa(H)-(1-2\varepsilon)\kappa(H) =2\varepsilon \kappa(H) \ge 2\varepsilon^2n=2\delta n$$ for every $i=1,\ldots,k$, where in the last step we used that $\kappa(H)\ge \varepsilon n$ by assumption. 
Consider for any fixed index $i \in [k]$ 
the set collection $(Z_h)_{h \in N_H(x_i) \cap X_B}$. Since the sets are pairwise disjoint and contained in the set $B$ of size at most $n$, as above it follows that at most $\delta n$ sets in this collection can be of size greater than $\frac{1}{\delta}$. Consequently, for each $i \in [k]$ there exists a subset $N_i \subseteq N_H(x_i) \cap X_B$ of size at least $2\delta n-\delta n=\delta n$ such that $|Z_h| \le \frac{1}{\delta}$ for every $h \in N_i$ and $i \in [k]$. Let $y_1,\ldots,y_l \in X_B$ be distinct vertices such that $\{y_1,\ldots,y_l\}=\bigcup_{i=1}^{k}{N_i}$. Then clearly, $l \ge |N_1|\ge \delta n$. Furthermore, we have $$\e_H(\{x_1,\ldots,x_k\},\{y_1,\ldots,y_l\})\ge \sum_{i=1}^{k}{|N_i|} \ge k\cdot \delta n\ge \delta^2 n^2\ge D n\log n,$$ where in the last step we used our assumption that $n \ge N \ge n_2$. 

   We can now use that $G$ satisfies property $\textsf{P}(H,\delta,\lceil Dn\log n\rceil)$, which directly implies that there exists a pair $(i,j) \in [k]^2$ such that $x_iy_j 
 \in E(H)$ and $G$ contains all edges of the form $xy$ where $(x,y) \in Z_{x_i}\times Z_{y_j}$. However, by definition of $F$ this means that there exists no edge in $F$ which has endpoints in both $Z_{x_i}$ and $Z_{y_j}$. This is a contradiction to our initial assumption that $(Z_h)_{h \in V(H)}$ form an $H$-minor model in $F$. Thus, $F$ does not contain $H$ as a minor, which establishes the third item of the lemma and concludes the proof.
\end{proof}

Our next lemma below guarantees that for sufficiently well-connected graphs $H$, the property of being $H$-minor-free is preserved when pasting together two graphs along a sufficiently small clique. This statement will then be used in the proof of Theorem~\ref{thm:conn} to glue several copies of the $H$-minor-free graph from Lemma~\ref{lemma:deterministiccorollary} along a common clique, thus eventually creating a graph that is still $H$-minor-free but has an increased list chromatic number. While the statement of the lemma is folklore in the graph minors community, for the sake of completeness we include its short proof. 

\begin{lemma}\label{lemma:glue}
Let $G_1, G_2$ be $H$-minor-free graphs and let $C:=V(G_1) \cap V(G_2)$. If $C$ forms a clique in both $G_1$ and $G_2$ and if $|C|<\kappa(H)$, then the graph union $G_1 \cup G_2$ is also $H$-minor-free.
\end{lemma}
\begin{proof}
    Towards a contradiction, suppose that $G:=G_1\cup G_2$ contains an $H$-minor model $(Z_h)_{h \in V(H)}$. Note that since every branch set induces a connected subgraph of $G$, for every $h \in V(H)$ with $Z_h \cap V(G_1) \neq \emptyset \neq Z_h \cap V(G_2)$ we also must have $Z_h \cap C \neq \emptyset$, since there are no edges between $V(G_1)\setminus C$ and $V(G_2)\setminus C$ in $G$. Thus, every vertex of $H$ belongs to exactly one of the following: 
    $$X_1:=\{h \in V(H)| Z_h \subseteq V(G_1) \setminus C\}, X_2:=\{h \in V(H)| Z_h \subseteq V(G_2) \setminus C\},$$ $$ X_C:=\{h \in V(H)| Z_h \cap C \neq \emptyset\}.$$ We claim that $X_1=\emptyset$ or $X_2=\emptyset$. Indeed, suppose towards a contradiction that $X_1, X_2\neq \emptyset$. Since the branch sets $(Z_h)_{h \in V(H)}$ are pairwise disjoint, we must have $|X_C|\le |C|<\kappa(H)$. Thus, $H-X_C$ is a connected graph, in particular there exists an edge $h_1h_2\in E(H)$ with $h_1 \in X_1$, $h_2 \in X_2$. However, this contradicts the fact that $(Z_h)_{h \in V(H)}$ is an $H$-minor model and that there can be no edges between $Z_{h_1}\subseteq V(G_1)\setminus C$ and $Z_{h_2}\subseteq V(G_2)\setminus C$ in $H$. This shows that indeed  $X_1$ or $X_2$ is empty, w.l.o.g. we can assume $X_2=\emptyset$.

    For every $h \in X_C$, define $Z_h':=Z_h \cap V(G_1)\neq \emptyset$. We now claim that the collection of disjoint non-empty sets $(Z_h)_{h \in X_1}$ together with $(Z_h')_{h \in X_C}$ forms an $H$-minor model in $G_1$. First of all, for every $h \in X_C$ we have that $G[Z_h']$ is a connected graph. This follows since $G[Z_h]$ is a connected graph and since any pair of vertices $x, y \in Z_h \cap V(G_1)$ that are connected by a path $P$ in $G[Z_h]$ are also connected by a path $P'$ in $G[Z_h']$ that is obtained from $P$ by replacing every of its segments starting and ending in a vertex of $C$ by a direct connection between its endpoints (recall that $C$ is a clique).
    
    It remains to verify that for every edge $h_1h_2\in E(H)$ there is a connection in $G_1$ between the branch sets of $h_1$ and $h_2$. If both $h_1, h_2 \in X_1$, this is obvious. If both $h_1, h_2 \in X_C$, then the fact that $Z_{h_1}\cap C \neq\emptyset \neq Z_{h_2}\cap C$ implies a connection between $Z_{h_1}'$ and $Z_{h_2}'$ in $G_1$ (witnessed by an edge in the clique $C$). Finally, if, say, $h_1 \in X_1$ and $h_2 \in X_C$, then consider an edge $xy \in E(G)$ with $x \in Z_{h_1}$, $y \in Z_{h_2}$. Since $Z_{h_1}\subseteq V(G_1) \setminus C$, the vertex $x$ has no connections to $V(G_2)\setminus C$ and thus we must have $y \in V(G_1)$. Thus the same edge $xy$ witnesses a connection between $Z_{h_1}$ and $Z_{h_2}'$ in $G_1$. All in all this shows that indeed, $G_1$ contains an $H$-minor model, which is a contradiction to the assumptions of the lemma. Thus, our initial assumption was wrong, $G$ is  $H$-minor-free.
\end{proof}

The last ingredient required to complete the proof of Theorem~\ref{thm:conn} is a simple but important idea on how to lower-bound the list chromatic of a graph that is obtained from a fixed graph $F$ by repeated pasting along the same clique. Since the statement will also be reused for the proof of Theorem~\ref{thm:random} in the next section, we decided to isolate it here. We use the following terminology:
\begin{definition}[Pasting]
    Let $F$ be a graph, let $S \subseteq V(F)$ and $K \in \mathbb{N}$. A \emph{$K$-fold pasting of $F$ at $S$} is any graph that can be expressed as the union of $K$ isomorphic copies $F_1, \ldots,F_K$ of $F$ with the property that $V(F_i) \cap V(F_j)=S$ for all $1 \le i < j \le K$.
\end{definition}
\begin{lemma} \label{lemma:cliqdegetc}
Let $m,n,d \in \mathbb{N}$ with $d \le m$ and let $F$ be a graph whose vertex set is partitioned into two cliques $A$, $B$ such that every vertex in $B$ has at least $|A|-d$ neighbors in $A$. 
Let $K = (|A|+|B|-1)^{|A|}$ and let $F^{(K)}$ be a $K$-fold pasting of $F$ at $A$. Then $\chi_\ell(F^{(K)}) \geq |A|+|B|-d$. 
\end{lemma}
\begin{proof}
Let $F_1,...,F_K$ be an ordering of the copies of $F$ in the pasting graph $F^{(K)}$, and let $B_1,...,B_K$ be the corresponding copies of $B$. 
Let $f : [|A|+|B|-1]^A \rightarrow [K]$ be an arbitrary bijection and let $c_1,...,c_K : A \rightarrow [|A|+|B|-1]$ be the ordering of color assignments to $A$ that satisfies $f(c_i) = i$ for all $i \in [K]$.
Consider the list assignment $L : V(F^{(K)}) \rightarrow 2^{[|A|+|B|-1]}$ defined as follows:
\begin{itemize}
    \item $L(a) := [|A|+|B|-1]$ for all $a \in A$
    \item $L(b) := [|A|+|B|-1] \setminus \{c_i(a) \mid a \in A \setminus N_{F_i}(b) \}$ for all $b \in B_i$ for all $i \in [K]$
\end{itemize}

Given that $B$ is a clique and every vertex in $B$ by assumption has at most $d$ non-neighbors in $F$, we have $|L(v)| \geq |A|+|B|-1-d$ for all $v \in V(F^{(K)})$. 
Now assume towards a contradiction that $F^{(K)}$ admits a proper $L$-coloring $c$ and let $i \in [K]$ be the unique index satisfying $c(a) = c_i(a)$ for all $a \in A$. 
Then let $c_{F_i} : A \cup B_i \rightarrow [|A|+|B|-1]$ be the coloring $c$ restricted to the graph $F_i$.
Since $\vv(F_i) = |A|+|B|$, there exist by the pigeonhole principle vertices $u,v \in V(F_i)$ with $c(u) = c(v)$. 
Since $c$ is proper and $A$, $B$ are cliques, we have $uv \notin E(F_i)$ and $u \in A$, $v \in B$ without loss of generality. 
However, $c(u) \notin L(v)$ by the construction of $L$, a contradiction.
\end{proof}

By assembling the previously established pieces, we can now easily deduce Theorem~\ref{thm:conn}.

\begin{proof}[Proof of Theorem~\ref{thm:conn}]
Let a constant $\varepsilon>0$ be given, and let $\tilde{\varepsilon}\in (0,1)$ be chosen small enough such that $3\tilde{\varepsilon}<\varepsilon$. Let $N=N(\tilde{\varepsilon})$ be as in Lemma~\ref{lemma:deterministiccorollary}. We now set $n_0:=\max\{N,\lceil\frac{1}{\varepsilon^2}\rceil,\lceil\frac{2}{\varepsilon-3\tilde{\varepsilon}}\rceil\}$ and claim that Theorem~\ref{thm:conn} holds for this choice of $n_0$. 

Let $H$ be a graph on $n \ge n_0$ vertices. We have to prove that $f_\ell(H)\ge (1-\varepsilon)(n+\kappa(H))$. If $\kappa(H)<\varepsilon n$, then this follows directly from the trivial lower bound via $$f_\ell(H) \ge \vv(H)-1=n-1\ge (1-\varepsilon^2)n=(1-\varepsilon)(n+\varepsilon n)>(1-\varepsilon)(n+\kappa(H)).$$
Thus, we may now assume $\kappa(H)\ge \varepsilon n$, in particular, $\kappa(H) \ge \tilde{\varepsilon} n$. Using $n \ge N$ and Lemma~\ref{lemma:deterministiccorollary} we now find that there exists an $H$-minor-free graph $F$ whose vertex-set is partitioned into two cliques $A, B$ such that $|A|=\lfloor (1-2\tilde{\varepsilon})\kappa(H)\rfloor <\kappa(H)$ and $|B|=\lfloor (1-2\tilde{\varepsilon})n\rfloor$, and such that every vertex in $B$ has at most $\tilde{\varepsilon}n$ non-neighbors in $F$. Let $d:=\lfloor \tilde{\varepsilon}n \rfloor$ and $K:=(|A|+|B|-1)^{|A|}$. Let $F^{(K)}$ denote a $K$-fold pasting of $F$ at the clique $A$. Since every vertex in $B$ has at least $|A|-d$ neighbors in $A$, we can apply Lemma~\ref{lemma:cliqdegetc} to find that 
$$\chi_\ell(F^{(K)}) \ge |A|+|B|-d \ge (1-2\tilde{\varepsilon})(\kappa(H)+n)-2-\tilde{\varepsilon}n$$
$$\ge (1-3\tilde{\varepsilon})(n+\kappa(H))-2\ge (1-\varepsilon)(n+\kappa(H)),$$
using $n \ge n_0 \ge \frac{2}{\varepsilon-3\tilde{\varepsilon}}$ in the last step. In addition, since $|A|<\kappa(H)$, the graph $F^{(K)}$ is $H$-minor-free by repeated application of Lemma~\ref{lemma:glue}. All in all, we conclude that $f_\ell(H)\ge (1-\varepsilon)(\vv(H)+\kappa(H))$, as desired.
\end{proof}

\smallskip
\section{Proof of Theorem~\ref{thm:random}}\label{sec:thm2proof}
In this section, we present the proof of Theorem~\ref{thm:random}. The theorem claims a lower bound on $f_\ell(H)$ for almost all graphs $H$ on $n$ vertices and $\lceil C n \log n\rceil$ edges for some large constant $C>0$. However, in fact the only condition on the graph $H$ our lower bound proof relies upon is the following pseudo-random graph property, guaranteeing the existence of many edges between every pair of disjoint linear-size vertex subsets in $H$.
\begin{definition}[Property $\textsf{Q}$, graph family $\mc Q_n$] \label{def:propP}
Let $\delta > 0$ and $D > 1$ be arbitrary. We say that a graph $H$ with $n$ vertices satisfies property $\textsf{Q}(\delta, D)$ if for every two disjoint vertex sets $A, B \subseteq V(H)$ with $|A|,|B| \geq \delta n$, we have $\e_H(A,B)\ge D n \log n$. Let $\mc{Q}_{n} (\delta, D)$ denote the family of $n$-vertex graphs $H$ that satisfy property $\textsf{Q}(\delta, D)$.
\end{definition}
Crucially, as we prove next, property $\textsf{Q}(\delta,D)$ is satisfied for almost all graphs on $n$ vertices with an average degree of $C\log n$ for a large enough constant $C$.
\begin{lemma} \label{lemma:propP}
Let $\delta > 0$, $D > 1$ be arbitrary and let $m:\mathbb{N}\rightarrow \mathbb{N}$ be defined as $m(n) = \lceil\frac{D^2}{ \delta ^2} n\log n\rceil$.
Then with high probability as $n \rightarrow \infty$, a random graph $H = G(n;m(n))$ drawn uniformly from all $n$-vertex graphs with $m(n)$ edges satisfies property $\textsf{Q}(\delta, D)$.
\end{lemma} 
\begin{proof}  
 One possible way to sample $H$ uniformly from all $n$-vertex graphs on $m=m(n)$ edges is as follows: Starting from the edgeless graph $H_0$ on vertex-set $[n]$, for $i=1,\ldots,m$ uniformly at random pick an edge $e_i$ not included in $H_{i-1}$ and add it to $H_{i-1}$, thus creating a new graph $H_i$. Then $(H_i)_{i=0}^{m}$ forms a sequence of random graphs on vertex-set $[n]$ where $H_i$ is drawn uniformly from the graphs with exactly $i$ edges. In particular, we may w.l.o.g. identify $H=H_m$. 

 Let us now go about proving the claim of the lemma. Let $A, B \subseteq [n]$ be any given pair of disjoint sets such that $|A|,|B|\ge \delta n$ and let us bound the probability that $\e_H(A,B)<Dn\log n$.
 We denote the number of potential edges $ab$ with $a\in A, b \in B$ by $N :=|A||B|\ge  \delta^2 n^2$. At the $i$-th step of the process above, the new edge $e_i$ which gets added to $H_{i-1}$ has endpoints in $A$ and $B$ with probability at least $\frac{N-(i-1)}{\binom{n}{2}-(i-1)}$.

For large enough $n$, we have $m=m(n)=\lceil \frac{D^2}{\delta^2}{n\log n}\rceil<\frac{\delta^2}{2}n^2$ and thus for $i=1,\ldots,m$ 
$$\mathbb{P}(\text{edge $e_i$ has endpoints in $A$ and $B$}) \ge \frac{N-(i-1)}{\binom{n}{2}-(i-1)} \ge \frac{N-\frac{\delta^2}{2}n^2}{\binom{n}{2}}>\delta^2 .$$

From this observation it follows that there exists a binomially distributed random variable $X\sim B(m(n),\delta^2)$ such that $\e_H(A,B)\ge X$. Let $\mu := \mathbb{E}[X] = m(n)\delta^2$. Using $\mu \ge D^2n\log n $ and Lemma~\ref{lemma:chernoff}, we have for large enough $n$ 
$$\mathbb{P}(\e_H(A,B) < D n \log n) 
\le \mathbb{P}\Bigl(X \leq \frac{\mu}{D} \Bigr)
\le \exp \Bigl(- \bigl(1-\frac{1}{D} \bigr)^2 \frac{\mu}{2} \Bigr)
\le \exp \Bigl(-  \frac{(D-1)^2}{2} n \log n \Bigr).
$$
The number of ways to choose the disjoint sets $A$ and $B$ can be bounded from above by $3^n$ (each vertex can be in $A$, $B$, or neither). 
By the union bound, we have
$$ \mathbb{P}(H \text{ violates property } \textsf{Q}(\delta,D)) \le 3^n \cdot \exp \Bigl(-  \frac{(D-1)^2}{2} n \log n \Bigr) \leq \exp(O(n) - \Omega(n \log n )).$$
This probability tends to $0$ as $n \rightarrow \infty$, as desired.
\end{proof}
In our next step towards proving Theorem~\ref{thm:random}, we establish the following statement somewhat analogous to Lemma~\ref{lemma:deterministiccorollary}, showing how to build small and close-to-complete $H$-minor-free graphs for a given graph $H \in \mathcal{Q}_n(\delta,D)$. 
\begin{lemma} \label{lemma:propQ-no-subgraph-minor}
    Let $\delta\in (0,1)$, $D > 1$, $n \in \mathbb N$, and $H \in \mc{Q}_{n}(\delta, D)$ be arbitrary. Moreover, let $G$ be a bipartite graph with bipartition $\{A, B\}$, $|A| = |B| = \lfloor(1-3\delta)n\rfloor$ satisfying property $\textsf{P}(H, \delta,s)$ for $s=\lceil Dn\log n\rceil$. Then its complement graph $G^\complement$ does not contain $H[U]$ as a minor for any $U \subseteq V(H)$ with $|U| \geq (1-\delta)n$.
\end{lemma}

\begin{proof}
    Assume $G^\complement$ contains $H[U]$ as a minor for some $U \subseteq V(H)$ with $|U| \geq (1-\delta)n$. 
    Let $(Z_h)_{h \in U}$ be an $H[U]$-minor model in $G^\complement$ and define
    $X_A := \{h \in U \mid Z_h \subseteq A\}$, 
    $X_B := \{h \in U \mid Z_h \subseteq B\}$, and 
    $X_{AB} := \{h \in U \mid Z_h \cap A \neq \emptyset \neq \ Z_h \cap B\}$.
    We have 
    $|X_A| + |X_{AB}| \leq |A|$, 
    $|X_B| + |X_{AB}| \leq |B|$, \\ and
    $|X_A| + |X_B| + |X_{AB}|=|U| \geq (1-\delta)n$, which implies $|X_A|, |X_B| \ge (1-\delta)n-(1-3\delta)n=2\delta n$. 
    \\ Since the branch-sets $(Z_h)_{h \in X_A}$ in $A$ and the branch-sets $(Z_h)_{h \in X_B}$ in $B$ are pairwise disjoint, at most $\delta (1-3\delta) n<\delta n$ branch sets in each of $(Z_h)_{h \in X_A}$ and $(Z_h)_{h \in X_B}$ can be larger than $\frac{1}{\delta}$. Thus, there are at least $2\delta n-\delta n=\delta n$ branch sets of size at most $\frac{1}{\delta}$ in $(Z_h)_{h \in X_A}$ as well as in $(Z_h)_{h \in X_B}$.
    Thus for $k:=l:= \lceil \delta n \rceil$, there exist distinct vertices $x_1, ..., x_{k} \in X_A$, $y_1,\ldots,y_l \in X_B$ such that $|Z_{x_i}|, |Z_{y_j}| \le \frac{1}{\delta}$ for all $1 \le i, j \le k=l$. Since $H\in \mathcal{Q}_n(\delta,D)$, we have $\e_H(\{x_1,\ldots,x_k\},\{y_1,\ldots,y_l\})\ge \lceil D n \log n\rceil=s$.   
    Next we use our assumption that $G$ satisfies property $\textsf{Q}(H, \delta,s)$. It implies that there exists an edge $x_i y_j \in E(H)$ with $(i,j) \in [k]\times [l]$ such that $G$ contains all the edges $xy$ with $(x,y) \in Z_{x_i} \times Z_{y_j}$. Then, however, there is an edge between vertices $x_i$ and $y_j$ in $H$, but no edge between the corresponding branch sets $Z_{x_i}$ and $Z_{y_j}$ in $G^\complement$, a contradiction. 
\end{proof}
The next auxiliary statement we need is Lemma~\ref{lemma:subgraph-minor} below, which establishes a weak analogue of Lemma~\ref{lemma:glue} for graphs $H\in \mathcal{Q}_n(\delta,D)$. Note that as these graphs may have sublinear minimum degree and connectivity, Lemma~\ref{lemma:glue} cannot be used to obtain the same statement.
\begin{lemma} \label{lemma:subgraph-minor}
    Let $\delta > 0$, $D > 1$, $H \in \mc{Q}_n(\delta, D)$ and let $F$ be a graph with a clique $W \subseteq V(F)$ of size $\lfloor(1-3\delta)n\rfloor$. Let $K
    \in \mathbb{N}$ and let $F^{(K)}$ be a $K$-fold pasting of $F$ at $W$.  If $F^{(K)}$ contains $H$ as a minor, then there exists $U \subseteq V(H)$ with $|U|\ge (1-\delta)n$ such that $F$ contains $H[U]$ as a minor. 
\end{lemma}
\begin{proof}
In the following, let $F_1, \ldots, F_K$ denote the copies of $F$ such that $F^{(K)}=\bigcup_{i=1}^{K}{F_i}$.

Suppose $F^{(K)}$ has an $H$-minor and fix an $H$-minor model $(Z_h)_{h\in V(H)}$ in $H$. Let us denote $X_W := \{h \in V(H) \mid Z_h \cap W \neq \emptyset\}$ and $\xi_W:=|X_W|$, and $X_i := \{h \in V(H) \mid Z_h \subseteq V(F_i) \setminus W\}$ and $\xi_i:=|X_i|$ for every $i \in [K]$. Note that since every branch-set $Z_h$ induces a connected subgraph of $F$, every vertex $h \in V(H)$ appears in exactly one of the sets $X_W, X_1,\ldots,X_K$, i.e., they form a partition of $V(H)$. In particular, we have $\xi_W+\sum_{i=1}^{K}{\xi_i}=\vv(H)=n$. 

We have $\xi_W \leq |W|\le  n-3\delta n$ and thus $\sum_{i=1}^{K} \xi_i=n-\xi_W \geq 3 \delta n$. In the following, let us w.l.o.g. assume $[K]$ is ordered such that $\xi_1 \geq \xi_2 \geq \cdots \geq \xi_K$. We claim that $\xi_1\ge (1-\delta)n-\xi_W$. 
Towards a contradiction, suppose in the following that $\xi_1 < (1-\delta)n - \xi_W$. We first note that using this assumption, we have that $\sum_{i=2}^K \xi_i=n-(\xi_W+\xi_1) >n-(1-\delta)n=\delta n$.

Now suppose for a first case that $\xi_1 \ge \delta n$. Then the two disjoint sets of vertices $X_1$ and $\bigcup_{i=2}^{K}{X_i}$ in $H$ are both of size at least $\delta n$. By property $\textsf{Q}(\delta,D)$ this implies that $\e_H(X_1,\bigcup_{i=2}^{K}{X_i})\ge Dn\log n>0$. In particular there exists $2 \le i \le K$ and an edge $uv\in E(H)$ for some $u \in X_1$ and $v \in X_i$. This implies that there must exist an edge in $F^{(K)}$ connecting a vertex in $Z_u\subseteq V(F_1)\setminus W$ to a vertex in $Z_v\subseteq V(F_i)\setminus W$. However, by construction of $F^{(K)}$ no such edges exist, and so we arrive at the desired contradiction in this first case. 

For the second case, suppose that $\xi_1<\delta n$ (and thus in particular $\xi_i<\delta n$ for all $i \in [K]$). Let $j \in [K]$ be the smallest index such that $\sum_{i=1}^{j}{\xi_i}>\delta n$ (this is well-defined, since $\sum_{i=1}^K{\xi_i}\ge 3\delta n$, see above). By the minimality of $j$, we have $\sum_{i=1}^{j}{\xi_i}=\xi_j+\sum_{i=1}^{j-1}{\xi_j}\le \delta n+\delta n=2\delta n$. This implies that $\sum_{i=j+1}^K{\xi_i}=\sum_{i=1}^K{\xi_i}-\sum_{i=1}^j{\xi_i}\ge 3\delta n - 2\delta n =\delta n$. In consequence, we find that the two disjoint vertex sets $\bigcup_{i=1}^{j}{X_i}, \bigcup_{i=j+1}^{K}{X_i}$ in $H$ are both of size at least $\delta n$. Hence, using property $\textsf{Q}(\delta,D)$ we have $\e_H(\bigcup_{i=1}^{j}{X_i},\bigcup_{i=j+1}^{K}{X_i})\ge Dn\log n>0$. Similar as above, this implies the existence of two indices $i,i'$ with $1 \le i\le j<i'\le K$ such that there exists an edge between $V(F_i)\setminus W$ and $V(F_{i'})\setminus W$ in $F^{(K)}$. As this is impossible by construction of $F^{(K)}$, a contradiction follows also in the second case. Thus our initial assumption $\xi_1<(1-\delta)n-\xi_W$ was false. 

We therefore have $|X_1 \cup X_W|=\xi_1+\xi_W \geq (1-\delta)n$. Let $U:=X_1 \cup X_W$. For every $h \in U$, let $Z_h':=Z_h$ if $h \in X_1$ and $Z_h':=Z_h \cap V(F_1)$ if $h \in X_W$. 
We now show that $(Z_h')_{h \in U}$ is an $H[U]$-minor model in $F_1$, which will then conclude the proof of the lemma.

First of all, note that $F_1[Z_h']$ is a connected graph for every $h \in U$. If $h \in X_1$, then $F_1[Z_h']=F^{(K)}[Z_h]$ is connected since $(Z_h)_{h \in V(H)}$ is an $H$-minor model. And if $h \in X_W$, then the connectivity of $F_1[Z_h']=F^{(k)}[Z_h \cap V(F_1)]$ follows since (1) $F^{(K)}[Z_h]$ is connected and (2) every path connecting two vertices in $Z_h'$ that is contained in $F^{(K)}[Z_h]$ can be shortened to a path whose vertex-set is completely contained in $V(F_1)$ by short-cutting every segment of the path that starts and ends in the clique $W$ by the direct connection between its endpoints.

Let us now consider any edge $uv \in E(H[U])$. Then there must exist an edge $xy \in E(F^{(K)})$ with $x \in Z_u, y \in Z_v$. If we have $x, y\in V(F_1)$, then this witnesses the existence of an edge between $Z_u'$ and $Z_v'$ in $F_1$, as desired. If on the other hand at least one of $x,y$ lies outside of $V(F_1)$, then we necessarily must have $Z_u \cap W \neq \emptyset \neq Z_v \cap W$, and thus there exists an edge in the clique induced by $W$ (and thus also in $F_1$) that connects a vertex in $Z_u'$ to a vertex in $Z_v'$. All in all, this shows that $F_1$ contains $H[U]$ as a minor. Since $|U| \ge (1-\delta)n$, this concludes the proof.
\end{proof}

With the previous auxiliary results at hand, we can now deduce Theorem~\ref{thm:random}.

\begin{proof}[Proof of Theorem~\ref{thm:random}]
Let a constant $\varepsilon \in (0,1)$ be given. Let $\delta>0$ be chosen small enough such that $7\delta <\varepsilon$, set $p:=\frac{\delta}{2}$, let $D=D(\delta,p)>1$ be the constant given by Lemma~\ref{lemma:randbip}, and let $C := \frac{D^2}{\delta^2}$. 
\\For every $n \in \mathbb{N}$, put $s=s(n)=\lceil Dn\log n\rceil$. 
By Lemma~\ref{lemma:propP}, a random graph $H = G(n; \lceil Cn \log n\rceil)$ chosen uniformly from all $n$-vertex graphs with $\lceil Cn\log n\rceil$ edges satisfies property $\textsf{Q}(\delta,D)$ w.h.p. as $n \rightarrow \infty$. 
Now assume the graph $H$ satisfies property $\textsf{Q}(\delta,D)$. 
By Lemmas~\ref{lemma:maxdeg} and \ref{lemma:randbip}, w.h.p. as $n \rightarrow \infty$, the random bipartite graph $G = G(\lfloor(1-3\delta)n\rfloor, \lfloor(1-3\delta)n\rfloor ; p )$ has maximum degree at most $2p\lfloor(1-3\delta)n\rfloor\le \delta n$ and satisfies property $\textsf{P}(H, \delta,s)$.
Now fix $n$ large enough and consider a graph $G$ with bipartition $\{A,B\}$, $|A|=|B|=\lfloor(1-3\delta n)\rfloor$ satisfying these two properties. 
By Lemma \ref{lemma:propQ-no-subgraph-minor}, $G^\complement$ does not contain any induced subgraph $H[U]$ as a minor for any $U \subseteq V(H)$ with $|U| \geq (1-\delta) n$. Let $K := (|A|+|B|-1)^{|A|}$ and let$(G^\complement)^{(K)}$ be a $K$-fold pasting of $G^\complement$ at $A$.
Then by Lemma~\ref{lemma:subgraph-minor}, $(G^\complement)^{(K)}$ does not contain $H$ as a minor. 
Moreover, by Lemma~\ref{lemma:cliqdegetc}, applied with $d=\lfloor \delta n\rfloor$, we find that $(G^\complement)^{(K)}$ has list chromatic number at least $|A|+|B|-d> 2(1-3\delta)n - \delta n -2 > (2-\varepsilon)n$ for $n$ large enough. This shows that w.h.p. the random graph $H=G(n;\lceil Cn \log n\rceil)$ satisfies $f_\ell(H)\ge (2-\varepsilon)n$, which concludes the proof.
\end{proof}

\smallskip
\section{Proof of Theorem~\ref{prop:isol}}\label{sec:append}
In this section we give the proof of Theorem~\ref{prop:isol}, which is self-contained and independent of the results in the previous sections. A basic tool from extremal graph theory used in the proof is Tur\'{a}ns theorem, in the following form:
\begin{theorem}[Tur\'{a}n]\label{thm:turan}
    Let $k \in \mathbb{N}$, $k \ge 2$ and let $G$ be a graph. If $\e(G)>(1-\frac{1}{k-1})\frac{\vv(G)^2}{2}$ then $G$ contains a clique on $k$ vertices.
\end{theorem}
We also use the following classical result regarding the minimum degree of $K_t$-minor-free graphs, as independently proved by Kostochka~\cite{kostochka} and Thomason~\cite{thomason}. 
\begin{theorem}[\cite{kostochka, thomason}]\label{thm:degeneracy}
For every integer $t \ge 1$ there exists an integer $d=d(t)=O(t \sqrt{\log t})$ such that every graph of minimum degree at least $d$ contains $K_t$ as a minor. In particular, for every graph $F$ there exists $d=d(F) \in \mathbb{N}$ such that all graphs of minimum degree at least $d$ contain $F$ as a minor.
\end{theorem}
\begin{proof}[Proof of Theorem~\ref{prop:isol}]
We start by fixing an integer $d \in \mathbb{N}$ as guaranteed by Theorem~\ref{thm:degeneracy}, i.e. such that every graph of minimum degree at least $d$ contains $F$ as a minor. We now define $k_0(F):=\min\{d+1,9\cdot\vv(F)^3\}$. Let $k \ge k_0(F)$ be any given integer. Let $H$ denote the graph obtained from $F$ by adding $k$ isolated vertices. We will now show that every $H$-minor-free graph is $(\vv(H)-2)$-degenerate, which then easily implies $f_\ell(H)=\vv(H)-1$. 

Towards a contradiction, suppose that there exists an $H$-minor-free graph $G$ which is not $(\vv(H)-2)$-degenerate, and let $G$ be chosen such that $\vv(G)$ is minimized. Note that the minimality assumption on $G$ immediately implies that $\delta(G)\ge \vv(H)-1=\vv(F)+k-1$. Observe that since $\delta(G)\ge k>d$, the graph $G-x$ for some $x \in V(G)$ has minimum degree at least $d$ and thus must contain $F$ as a minor. Let $X \subseteq V(G)$ be chosen of minimum size subject to $G[X]$ containing $F$ as a minor. Note that from the above it follows that $|X|\le\vv(G)-1$ and hence that $V(G)\setminus X \neq \emptyset$. Let $(Z_f)_{f \in V(F)}$ be an $F$-minor model in $G[X]$. By minimality of $X$, we have that $(Z_f)_{f \in V(F)}$ forms a partition of $X$. With the goal of bounding the number of edges in $G-X$, we present our next argument as a separate claim.

\begin{claim} \label{claim}
For every $v \in V(G)\setminus X$ and every $f \in V(F)$, we have $|N(v) \cap Z_f| < 9\vv(F)$.
\end{claim}
\begin{proof}
For $|Z_f|=1$ the inequality $|N(v) \cap Z_f|\le 1<9\vv(F)$ trivially holds for every $v \in V(G)\setminus X$. We may therefore assume $|Z_f|\ge 2$. Let $T_f$ denote a spanning tree of the connected graph $G[Z_f]$, and let $L_f\subseteq Z_f$ be the set of leaves in $T_f$. Then for every $l \in L_f$ the graph $G[Z_f\setminus \{l\}]$ is still connected. However, by minimality of $X$, $G[X]-l$ does not contain $F$ as a minor, and thus in particular the set system consisting of $Z_f\setminus \{l\}$ together with the remaining branch-sets $(Z_{f'})_{f' \in V(F), f' \neq f}$ cannot be an $F$-minor model in $G$. In consequence, there has to exist some $f' \in V(F)\setminus\{f\}$ such that among all vertices in $Z_f$, the vertex $l$ is the only one that has a neighbor in $Z_{f'}$. Since the above argument applies to any choice of $l \in L_f$, and since the respective elements $f'$ have to be distinct for different choices of $l$, it follows that $|L_f|\le |V(F)\setminus \{f\}|= \vv(F)-1$. Hence $T_f$ is a tree with at most $\vv(F)-1$ leaves. Let $T_f'$ be a tree without degree $2$-vertices such that $T_f$ is a subdivision of $T_f'$, i.e., every edge in $T_f'$ corresponds to one maximal path of $T_f$ all whose internal vertices are of degree $2$. Then, since $T_f'$ is a tree and thus has average degree less than $2$, it contains less vertices of degree at least $3$ than it has leaves. In particular, $\vv(T_f')\le |L_f|+(|L_f|-1)\le 2\vv(F)-3$ and therefore $\e(T_f')=\vv(T_f')-1\le 2\vv(F)-4<2\vv(F)$. This means that $T_f$ can be expressed as the edge-disjoint union of a collection of paths $(P_i)_{i=1}^r$ where $r<2\vv(F)$ and the internal vertices of each path $P_i$ are of degree $2$ in $T_f$. 

Next let us pick some set $Y\subseteq Z_f$ of size at most $\vv(F)-1$ as follows: For each edge $ff' \in E(F)$, pick some vertex $y_{f'} \in Z_f$ that has at least one neighbor in $Z_{f'}$ and add it to $Y$. Let $\mathcal{R}$ denote the collection of internally disjoint paths in $T_f$ obtained from $(P_i)_{i=1}^r$ by splitting each path $P_i$ into its maximal subpaths that do not contain internal vertices in $Y$. It is easy to see that $|\mathcal{R}| \le r+|Y|<2\vv(F)+\vv(F)=3\vv(F)$, and that $T_f$ equals the union of the paths in $\mathcal{R}$. We next claim that for every vertex $v \in V(G)\setminus X$ and every $R \in \mathcal{R}$, we have $|N(v) \cap V(R)|\le 3$. Indeed, suppose that $v$ hast at least $4$ distinct neighbors on $R$. Let $x$ and $y$ be the two neighbors of $v$ on $R$ that are closest to the endpoints of $R$. Define $R'$ as the path obtained from $R$ by replacing its subpath between $x$ and $y$ (which has to contain at least two internal vertices) by the path $x-v-y$ of length two. Let $A$ be the set of vertices on $R$ strictly between $x$ and $y$. Setting $X':=(X\setminus A) \cup \{v\}$ we can see that $|X'|<|X|$. However, we can find an $F$-minor in $G[X']$, witnessed by the branch-sets $(Z_f\setminus A) \cup \{x\}$ together with $(Z_{f'})_{f' \in V(F), f'\neq f}$. Notice that $G[(Z_f\setminus A) \cup \{x\}]$ is indeed connected, since all the internal vertices of $R$ were of degree $2$ in $T_f$. Also, since $Y\subseteq Z_f\setminus A$, it is still true that there exists a connection from a vertex in $(Z_f\setminus A) \cup \{x\}$ (namely, $y_{f'}$) to a vertex in $Z_{f'}$ for every edge $ff'\in E(F)$. Finally, this contradicts our initial choice of $X$ and proves that our assumption was wrong, so indeed every vertex $v \in V(G)\setminus X$ satisfies $|N(v) \cap V(R)|\le 3$ for every $R\in \mathcal{R}$. Therefore, we have $|N(v) \cap Z_f| \le \sum_{R \in \mathcal{R}}{|N(v) \cap V(R)|}\le 3|\mathcal{R}|<9\vv(F)$ for every $v \in V(G)\setminus X$, which concludes the proof of the claim.
\end{proof}

It follows immediately from Claim \ref{claim} that $|N(v) \cap X| \le \sum_{f \in V(F)}{|N(v) \cap Z_f|} <9\vv(F)^2$ for every $v \in V(G)\setminus X$. Additionally recalling that $\delta(G) \ge \vv(F)+k-1\ge k$, we find that for every $v \in V(G)\setminus X$, we have $\text{deg}_{G-X}(v)=|N(v) \setminus X|=\text{deg}(v)-|N(v) \cap X|>k-9\vv(F)^2$. Having established $V(G)\setminus X \neq \emptyset$ at the beginning of the proof, it now follows that $G-X$ is a graph of minimum degree greater than $k-9\vv(F)^2$. Also, note that since $G[X]$ contains $F$ as a minor, we are not able to find $k$ distinct vertices in $V(G)\setminus X$ as these could be used to augment the $F$-minor in $G[X]$ to an $H$-minor in $G$, contradicting our assumptions. We thus have $\vv(G-X)<k$. Using our choice of $k_0$ and $k \ge k_0$, it now follows that 
$$\delta(G-X) >k-9\vv(F)^2> \left(1-\frac{1}{\vv(F)-1}\right)k > \left(1-\frac{1}{\vv(F)-1}\right)\vv(G-X).$$ 
Therefore, $G-X$ has more than $\bigl(1-\frac{1}{\vv(F)-1}\bigr)\frac{\vv(G-X)^2}{2}$ edges and thus Theorem~\ref{thm:turan} implies the existence of a clique on $\vv(F)$ vertices in $G-X$. In particular, $G-X$ and thus $G$ contain a subgraph isomorphic to $F$. Let $K \subseteq V(G)$ be the vertex-set of such a copy of $F$. Then, since $\vv(G)\ge \delta(G)+1\ge \vv(F)+k$, there are at least $k$ vertices outside of $K$ in $G$, which can be added to the copy of $F$ on vertex-set $K$ to create a subgraph of $G$ that is isomorphic to $H$. In particular, this means that $G$ contains $H$ as a minor, a contradiction. All in all, we find that our initial assumption, namely regarding the existence of a smallest counterexample $G$ to our claim, was wrong. This concludes the proof that all $H$-minor-free graphs are $(\vv(H)-2)$-degenerate. 

It is a well-known fact and easy to prove by induction that for every $a \in \mathbb{N}$ all $a$-degenerate graphs are $(a+1)$-choosable. Thus what we have proved also implies that every $H$-minor-free graph is $(\vv(H)-1)$-choosable, as desired. All in all, it follows that $f_\ell(H)=\vv(H)-1$, concluding the proof of the theorem.
\end{proof}

\end{document}